\documentclass{amsart}

\usepackage[
  left=2.2cm,
  right=2.2cm,
  top=2.5cm,
  bottom=2.5cm
]{geometry}

\usepackage{amsmath}
\usepackage{amssymb}
\usepackage{amsthm}

\usepackage{graphicx}
\usepackage{caption}
\usepackage{subcaption}
\usepackage{epstopdf}

\usepackage{algorithm}
\usepackage{algpseudocode}
\usepackage{float}

\usepackage{xcolor}

\newtheorem{theorem}{Theorem}[section]
\newtheorem{lemma}[theorem]{Lemma}
\newtheorem{defn}[theorem]{Definition}
\theoremstyle{remark}
\newtheorem{remark}[theorem]{Remark}

\begin{document}

\title[Kernel TM Systems in Boundary Weighted Hardy Spaces]
{Kernel Takenaka-Malmquist Systems and Adaptive Approximation in Boundary Weighted Hardy Spaces}

\author{Tao Qian}
\address{
Macau Center for Mathematical Sciences,
Macau University of Science and Technology,
Macao, China
}
\email{tqian@must.edu.mo}

%\thanks{Corresponding authors: Tao Qian}


\date{}

\begin{abstract}
Takenaka-Malmquist (TM) systems, as generalizations of the trigonometric function system, play a decisive role in approximation theory on the open unit disc and the upper-half complex plane. In the present paper with the setting of boundary weighted Hardy spaces we develop what we call kernel-TM systems and related approximation theory aspects. In such framework, besides adaptive Fourier decomposition (AFD), we also establish double-zero AFD (DAFD)as an advanced sparse representation in the generalized Hardy spaces. The paper also includes brief reviews of several commonly studied RKHSs as subclasses of the classical Hardy space with particular attention to admissibility of kernel-TM systems.
\end{abstract}

\maketitle

\section{Introduction}
\def\bD{\bf D}

In this section, we give a short survey of existing knowledge on classical Hardy spaces and adaptive TM system decompositions. Such decompositions are also customarily called AFD, standing for \emph{adaptive Fourier decomposition}.

AFD and Double Zero AFD (DAFD), as abbreviations of adaptive Fourier decomposition and double-zero adaptive Fourier decomposition, respectively, will be briefly recalled. For a more detailed survey of these subjects, the reader is referred to \cite{QWQW}. In the present paper, we mainly treat the Hardy space in the open unit disc. The upper-half complex plane case is similar, with non-essential but necessary modifications.

Let $H^2(\bD)$, or simply $H^2$, denote the Hardy space in the unit disc ${\bf D}$. Among a number of equivalent definitions, we use the one in terms of Taylor series expansions:
\[
H^2({\bD})
=
\left\{
f:\ f\ \text{is holomorphic with}\ 
f(z)=\sum_{k=0}^{\infty}c_kz^k,\quad
\sum_{k=0}^{\infty}|c_k|^2<\infty
\right\}.
\]

It is well known that $H^2({\bD})$ is an RKHS with reproducing kernel
\[
k_a(z)=\frac{1}{1-\overline{a}z},
\]
whose norm-one normalization is
\[
e_a(z)=\frac{\sqrt{1-|a|^2}}{1-\overline{a}z}.
\]
For a finite sequence $(a_1,\ldots,a_n)$ or an infinite sequence
$(a_1,\ldots,a_n,\ldots)$ of complex numbers in the unit disc, allowing repetitions, one defines a finite or infinite TM system
\[
\{B_k\}_{k=1}^n
\quad\text{or}\quad
\{B_k\}_{k=1}^{\infty},
\]
where
\[
B_k(z)
:=
\frac{\sqrt{1-|a_k|^2}}{1-\overline{a_k}z}
\prod_{l=1}^{k-1}\frac{z-a_l}{1-\overline{a_l}z}
=
e_{a_k}(z)\prod_{l=1}^{k-1}\tau_{a_l}(z)
\]
with
\[
\tau_{a_l}(z)=\frac{z-a_l}{1-\overline{a_l}z}
\]
being \emph{the canonical M\"obius transform with pole} $a_l$.

Let $\tilde{k}_{a_l}$ denote the multiple Szeg\"o kernel:
\begin{eqnarray}
\tilde{k}_{a_l}
=
\left[
\left(
\frac{\partial}{\partial\overline{a}}
\right)^{L(a_l)-1}
k_a
\right]_{a=a_l},
\end{eqnarray}
where $L(a_l)$ is the multiplicity, that is, the number of repetitions of $a_l$ in the finite sequence $(a_1,\ldots,a_l)$. Denote by $\tilde{e}_{a_l}$ the norm-one normalization of $\tilde{k}_{a_l}$.

In \cite{QSAFD} (see its Appendix), it is shown that $\{B_k\}_{k=1}^{\infty}$ is obtained by consecutively applying the Gram--Schmidt orthonormalization process to the sequence of multiple Szeg\"o kernels
\[
(\tilde{k}_{a_1},\ldots,\tilde{k}_{a_l},\ldots)
\]
corresponding to the sequence
\[
{\bf a}=(a_1,\ldots,a_l,\ldots).
\]
We sometimes indicate the dependence of $\{B_k\}$ on ${\bf a}$ by writing it as $\{{\ }^{\bf a}B_k\}$. It is basic knowledge in the theory of TM systems and Blaschke products that
\[
\{{\ }^{\bf a}B_k\}_{k=1}^{\infty}
\]
is an orthonormal system in $H^2({\bD})$ and is a basis of $H^2({\bD})$ if and only if
\[
\sum_{k=1}^{\infty}(1-|a_k|)=\infty,
\]
which is the non-Blaschke condition. The Fourier basis system
\[
\{z^l\}_{l=0}^{\infty}
\]
is a particular case. In \cite{QCT}, it is shown that TM systems are Schauder bases. For more general knowledge of TM systems, we refer the reader to \cite{Gar}.

Denote by $P_{\{a_1,\ldots,a_l\}}$ the orthogonal projection onto the span of
\[
\{\tilde{k}_{a_1},\ldots,\tilde{k}_{a_l}\},
\]
and let
\[
Q_{\{a_1,\ldots,a_l\}}
=
I-P_{\{a_1,\ldots,a_l\}}
\]
be the orthogonal projection onto the orthogonal complement of the span of $\{\tilde{k}_{a_1},\ldots,\tilde{k}_{a_l}\}$.

Let $f\in H^2({\bD})$ be fixed, and assume that $f_1=f$. Let $a_1,\ldots,a_k$ be given or selected. Then the $(k+1)$-th \emph{reduced remainder} $f_{k+1}$ is defined through the $a_k$-\emph{generalized backward shift operator} by
\begin{eqnarray}
\label{stead1}
f_{k+1}(z)
=
\frac{
f_k(z)-\langle f_k,e_{a_k}\rangle e_{a_k}(z)
}{
\frac{z-a_k}{1-\overline{a_k}z}
}
=
\frac{
Q_{\{e_{a_k}\}}(f_k)(z)
}{
\tau_{a_k}(z)
}
\in H^2({\bD}).
\end{eqnarray}
To generate AFD, $a_k$ is selected optimally according to
\begin{eqnarray}
\label{under1}
a_k
=
\arg\max
\left\{
|\langle f_k,e_a\rangle|:\ a\in{\bD}
\right\}.
\end{eqnarray}
The facts that, at every step, an optimal parameter $a_k$ is attainable in ${\bD}$, and  $f_{k+1}$ belongs to the Hardy space, and that the following expansion converges in the norm sense are all  established in \cite{QWa1}: 
\begin{eqnarray}
\label{AFD1}
f(z)
=
\sum_{k=1}^{\infty}
\langle f_k,e_{a_k}\rangle
{\ }^{\bf a}B_k(z),
\end{eqnarray}
where the TM system $\{{\ }^{\bf a}B_k\}$ is accordingly generated by the selected parameter sequence ${\bf a}=\{a_k\}$. The convergence-rate theory for greedy algorithms is applicable to AFD; see \cite{DT}. Besides the fast convergence resulting from the Maximal Selection Principle (MSP) \eqref{under1}, the advantages of AFD include the automatically obtained orthonormality of the expanding TM system and, furthermore, the increase in the phase derivatives, or instantaneous frequencies, of the terms in the resulting functional series. We also note the useful relation
\[
\langle f_k,e_{a_k}\rangle
=
\langle f,{\ }^{\bf a}B_k\rangle,
\]
which follows from orthogonality.

There is also an associated interpolation-type result. For any sequence
\[
{\bf b}=(b_1,\ldots,b_n,\ldots)
\]
in $\bD$, allowing repetitions, the following algebraic identity holds:
\begin{eqnarray}
\label{inter1}
f(z)
&=&
\sum_{k=1}^{n}
\langle f,{\ }^{\bf b}B_k\rangle
{\ }^{\bf b}B_k(z)
+
{\ }^{\bf b}f_{n+1}(z)
\prod_{l=1}^{n}\tau_{b_l}(z)
\nonumber\\
&=&
P_{\{b_1,\ldots,b_n\}}(f)(z)
+
{\ }^{\bf b}f_{n+1}(z)
\prod_{l=1}^{n}\tau_{b_l}(z).
\end{eqnarray}
This identity shows that the rational function
\[
P_{\{b_1,\ldots,b_n\}}(f)
\]
interpolates the given Hardy space function $f$ at the points $b_1,\ldots,b_n$. For every $j$ between $1$ and $n$, one has
\[
\begin{aligned}
f(b_j)
&=
\langle f,k_{b_j}\rangle\\
&=
\left\langle
P_{\{b_1,\ldots,b_n\}}(f)
+
{\ }^{\bf b}f_{n+1}
\prod_{l=1}^{n}\tau_{b_l},
k_{b_j}
\right\rangle\\
&=
P_{\{b_1,\ldots,b_n\}}(f)(b_j).
\end{aligned}
\]
This interpolation property, however, is solely a consequence of the reproducing property of $k_{b_j}$. It especially has nothing to do with the optimal parameter selection. The role of the latter is to enhance the convergence speed toward the originally given function $f$.

A recent study, called Double Zero AFD, or simply DAFD, \cite{QWQW}, marks a new milestone in the development of AFD-type sparse approximation. The new observation is that, under the optimal selections \eqref{under1}, one can improve upon \eqref{stead1}, and consequently upon \eqref{AFD1}. We explain this in detail.

Given $f\in H^2({\bD})$, it can be shown that there exists a sequence of functions $\tilde{f}_{k+1}$ in the Hardy space such that $\tilde{f}_1=f$ and
\begin{eqnarray}
\label{new}
\tilde{f}_{k+1}(z)
=
\frac{
\tilde{f}_k(z)
-
\langle\tilde{f}_k,e_{\tilde{a}_k}\rangle
e_{\tilde{a}_k}(z)
}{
\left[
\frac{z-\tilde{a}_k}
{1-\overline{\tilde{a}_k}z}
\right]^2
}
=
\frac{
Q_{\{{\tilde{a}_k}\}}(\tilde{f}_k)(z)
}{
\tau_{\tilde{a}_k}^2(z)
}
\in H^2({\bD}),
\end{eqnarray}
where
\begin{eqnarray}
\label{uunder}
\tilde{a}_k
=
\arg\max
\left\{
|\langle\tilde{f}_k,e_a\rangle|:\ a\in{\bD}
\right\}.
\end{eqnarray}
The orthogonal complement projection
\[
Q_{\{{\tilde{a}_k}\}}
\]
alone gives rise to a zero of order one at $\tilde{a}_k$, while the maximality of $\tilde{a}_k$ increases the order of the zero by one. This induces a new and more efficient decomposition of $f$ in terms of the double-TM system generated by
\[
\tilde{\bf a}
=
(\tilde{a}_1,\ldots,\tilde{a}_k,\ldots):
\]
\begin{eqnarray}
\label{DAFD}
f(z)
=
\sum_{k=1}^{\infty}
\langle\tilde{f}_k,e_{\tilde{a}_k}\rangle
{\ }^{\tilde{\bf a}}\tilde{B}_k(z),
\end{eqnarray}
where a general element of the double-TM system
\[
\{
{\ }^{\tilde{\bf a}}\tilde{B}_k
\}_{k=1}^{\infty}
\]
is given by
\[
{\ }^{\tilde{\bf a}}\tilde{B}_k(z)
=
e_{\tilde{a}_k}(z)
\prod_{l=1}^{k-1}
\tau_{\tilde{a}_l}^2(z).
\]
A double-TM system may be regarded as a particular TM system and is therefore orthonormal. The decomposition of $f$ obtained in \eqref{DAFD} is called the \emph{Double Zero AFD}, or simply the \emph{DAFD}, of $f$.

There is also a corresponding interpolation result. For any finite section of $\tilde{\bf a}$, one has
\begin{eqnarray}
\label{Hermit}
f(z)
&=&
\sum_{k=1}^{n}
\langle f,{\ }^{\tilde{\bf a}}\tilde{B}_k\rangle
{\ }^{\tilde{\bf a}}\tilde{B}_k(z)
+
\tilde{f}_{n+1}(z)
\prod_{l=1}^{n}\tau_{\tilde{a}_l}^2(z)
\nonumber\\
&=&
P_{\{\tilde{a}_1,\tilde{a}_1,\ldots,
\tilde{a}_{n-1},\tilde{a}_{n-1},\tilde{a}_n\}}(f)(z)
+
\tilde{f}_{n+1}(z)
\prod_{l=1}^{n}\tau_{\tilde{a}_l}^2(z),
\end{eqnarray}
where $\tilde{f}_{n+1}$ is the $(n+1)$-st-order double-reduced remainder in the Hardy space.

We note that, although \eqref{inter1} holds for any finite sequence ${\bf b}$, \eqref{Hermit} holds only for the optimally selected sequence $\tilde{\bf a}$. To understand \eqref{Hermit}, we review the iteration. Let $\tilde{a}_1=a_1$ be optimally selected according to \eqref{uunder}. As proved in \cite{QWQW}, $f_2(a_1)=0$. We take
\[
\tilde{a}_2=\tilde{a}_1=a_1
\]
without considering optimality again. Then, since
\[
\tilde{f}_1=f_1=f
\quad\text{and}\quad
f_2(\tilde{a}_2)=f_2(a_1)=0,
\]
we have
\[
\begin{aligned}
\tilde{f}_2(z)
&:=
\frac{
\tilde{f}_1(z)
-
\langle\tilde{f}_1,e_{\tilde{a}_1}\rangle
e_{\tilde{a}_1}(z)
}{
\tau_{\tilde{a}_1}^2(z)
}\\
&=
\frac{
f_1(z)
-
\langle f_1,e_{a_1}\rangle e_{a_1}(z)
}{
\tau_{a_1}^2(z)
}\\
&=
\frac{
f_2(z)
-
\langle f_2,e_{\tilde{a}_2}\rangle
e_{\tilde{a}_2}(z)
}{
\tau_{\tilde{a}_2}(z)
}.
\end{aligned}
\]

The above iteration involves repeated selections of certain parameters and, as a result,
\[
\sum_{k=1}^{n}
\langle f,{\ }^{\tilde{\bf a}}\tilde{B}_k\rangle
{\ }^{\tilde{\bf a}}\tilde{B}_k(z)
=
P_{\{\tilde{a}_1,\tilde{a}_1,\ldots,
\tilde{a}_{n-1},\tilde{a}_{n-1},\tilde{a}_n\}}(f)(z).
\]
The last equality can also be understood from the fact that the span of the finite TM system generated by the multiple Szeg\"o kernels associated with the parameter sequence
\[
(\tilde{a}_1,\tilde{a}_1,\ldots,
\tilde{a}_{n-1},\tilde{a}_{n-1},\tilde{a}_n)
\]
is identical to the orthogonal complement of the zero space
\[
\left(
\prod_{l=1}^{n}\tau_{\tilde{a}_l}^2
\right)
H^2({\bD}).
\]

The identity \eqref{Hermit} implies that the rational function
\[
P_{\{\tilde{a}_1,\tilde{a}_1,\ldots,
\tilde{a}_{n-1},\tilde{a}_{n-1},\tilde{a}_n\}}(f)
\]
interpolates both the Hardy space function $f$ and its first-order derivative at the optimally selected points
\[
\tilde{a}_1,\ldots,\tilde{a}_n
\]
under \eqref{uunder}. The numerical examples given in \cite{QWQW} show that the reconstruction speed of \eqref{DAFD} is approximately twice that of \eqref{AFD1}.

The advantages of DAFD include the following. First, the two algorithms have the same level of computational complexity, since both require the selection of $n$ optimal parameters, whereas in DAFD the approximating rational function is the orthogonal projection onto a larger space whose dimension is approximately twice that used in AFD. Second, DAFD provides a Hermite-type interpolation property that leads to an overall improvement in approximation. Third, from the viewpoint of digital signal processing, the remainders in DAFD possess zeros of higher order. This last point is also conceptually related to unwinding Blaschke expansions; see \cite{CS,Qinner}.

AFD-type sparse representations rely heavily on TM systems in Hardy spaces. Classical TM systems, which possess particularly simple and elegant constructive features, are obtained by applying the Gram--Schmidt process to multiple Szeg\"o kernels.

There have been attempts to construct TM systems for one and several complex variables, in other classical and non-classical domains, and for complex- or matrix-valued functions. To the knowledge of the author, these efforts have not yet produced demonstrative results. The Gram--Schmidt process applied to kernels in other RKHSs is generally highly complicated and therefore difficult to implement.

The aim of the present paper is to extend TM systems and related approximation subjects from Hardy spaces to what we call \emph{boundary weighted Hardy spaces} (BWHSs). We establish the associated boundary weighted AFD and boundary weighted double-zero AFD algorithms, abbreviated as BW-AFD and BW-DAFD, respectively.

Various aspects of BWHSs have been investigated by a number of authors, including H. Helson and G. Szeg\"o \cite{HS}, A. Hartmann and K. Seip, D. Sarason, D. Hitt, P. G\'erard, and A. Pushnitski. Such spaces have appeared, without a standard name for the function space, in the literature on Hardy spaces, Toeplitz and Hankel operators, model spaces, invariant subspaces, and related topics. The existing literature, however, has paid relatively little attention to the reproducing-kernel aspect.

In this paper, instead of applying the Gram--Schmidt process, we use an outer-function technique to construct TM-like systems. We simply call the resulting system the TM system of $H^2(m)$, where $m$, satisfying a two-sided boundedness condition, is the weight function defining the space.

In \S 2, we define boundary weighted Hardy spaces and prove a number of their properties. In \S 3, we establish TM systems in BWHSs and the associated adaptive, or sparse, TM representation, namely BW-AFD. In \S 4, we develop the corresponding double-zero AFD theory, namely BW-DAFD. In \S 5, we examine some RKHSs as subclasses of $H^2({\bD})$ for which a kernel-TM system constructor is unavailable.  

\section{Boundary Weighted Hardy Spaces}

As a general notation, let $H_K=H_K({\bD})$ denote an RKHS with reproducing kernel $K_w$. Denote the norm-one normalization of $K_w$ by
\[
E_w=\frac{K_w}{\|K_w\|_{H_K}}.
\]
The particular case $H_{k_w}=H^2({\bD})$ corresponds to $k_w$ being the Szeg\"o kernel. The norm-one normalization of $k_w$ is denoted by $e_w$.

We mainly work with what we call a \emph{boundary weighted Hardy space} (BWHS), defined through an absolutely continuous measure
\[
d\mu(t)=m(t)\,dt
\]
on the unit circle $\partial{\bD}$, where $m$ is a Lebesgue measurable function bounded both above and below by positive constants:
\[
0<c\leq m\leq C<\infty
\quad \text{a.e.},
\]
and
\[
H^2(m)
=
\overline{\{\text{holomorphic polynomials}\}}^{L^2(m\,dt)}.
\]
For basic knowledge of $H^2(m)$, or alternatively $H_{K^m},$ or BWHS, we refer the reader to the classical reference \cite{HS}.

The basic properties of BWHSs that we will concern are contained in the following theorem.

\begin{theorem}\label{Th1}
A boundary weighted Hardy space $H^2(m)$ satisfies the following properties.

\medskip

\noindent
{\rm (I)} In the set-theoretic sense,
\[
H^2(m)=H^2.
\]
Furthermore, the two Hilbert spaces have equivalent norms. As a consequence, the point-evaluation functionals on $H^2(m)$ are bounded, implying that $H^2(m)$ is an RKHS. We denote its reproducing kernel by
\[
K^m_w(z)=K^m(z,w).
\]

\medskip

\noindent
{\rm (II)} The space $H^2(m)$ is invariant under multiplication by finite Blaschke products. More precisely, if $B$ is a finite Blaschke product and $f\in H^2(m)$, then
\[
Bf\in H^2(m).
\]

\medskip

\noindent
{\rm (III)} The norm and the inner product are given by well-defined integrals on the unit circle with respect to the measure
\[
d\mu(t)=m(t)\,dt.
\]
That is,
\begin{eqnarray}\label{pol}
\langle f,g\rangle_{H^2(m)}
=
\int_0^{2\pi}
f^\ast(e^{it})
\overline{g^\ast(e^{it})}
m(t)\,dt,
\end{eqnarray}
where $f^\ast$ and $g^\ast$ are, respectively, the non-tangential boundary limit functions of
\[
f,g\in H_{K^m}.
\]
We will use the simplified notations
\[
\langle\cdot,\cdot\rangle_m
\quad\text{and}\quad
\|\cdot\|_m.
\]

\medskip

\noindent
{\rm (IV)} The space $H^2(m)$ is an RKHS with reproducing kernel
\[
K^m_w(z)
=
\frac{1}{h(z)\overline{h(w)}}
\frac{1}{1-z\overline{w}},
\]
where $h$ is an outer function in ${\bD}$, and the values $|h(z)|$, $z\in{\bD}$, are uniformly bounded above and below by positive constants. Furthermore, the \emph{boundary vanishing condition} (BVC) holds; that is,
\[
\lim_{|w|\to1}
|\langle f,E^m_w\rangle_m|
=
0.
\]
\end{theorem}

\noindent
{\bf Proof.}
Since $m$ is bounded below by a positive constant, a sequence of holomorphic polynomials $p_n$ can converge to $f$ in $L^2(m)$ only if $p_n$ converges to $f$ in $L^2(\partial{\bD})$. Thus,
\[
f\in H^2(m)
\]
implies that $f\in H^2$, and
\[
\|f\|_{H^2}\leq C\|f\|_m.
\]
By the two-sided boundedness of $m$, the converse argument also holds. Consequently, the two sets $H^2(m)$ and $H^2$ coincide, although they are equipped with different but equivalent norms. This proves (I).

Property (II) states that $H^2(m)$ is invariant under multiplication by finite Blaschke products. In fact, this is a consequence of (I). Since $H^2$ is invariant under multiplication by finite Blaschke products,
\[
f\in H^2(m)=H^2
\]
implies
\[
fB\in H^2=H^2(m).
\]

Property (III) follows from the dominated convergence theorem. We now prove (IV). The two-sided almost-everywhere boundedness of $m$ implies that
\[
\log m\in L^1(\partial{\bD}).
\]
Therefore, there exists an outer function $h$ such that
\[
|h^\ast(e^{it})|^2=m(t)
\]
for almost every $t\in[0,2\pi)$. See \cite{Du}. The outer function $h$ is given by
\[
h(z)
=
\exp\left\{
\frac{1}{4\pi}
\int_0^{2\pi}
\frac{e^{it}+z}{e^{it}-z}
\log m(t)\,dt
\right\},
\]
and is bounded above and below by positive constants for all $z\in{\bD}$.

The point-evaluation functionals on $H^2(m)$ are continuous, and hence $H^2(m)$ is an RKHS. We next show that
\begin{eqnarray}\label{ker}
K^m_w(z)
=
\frac{1}{h(z)\overline{h(w)}}
\frac{1}{1-\overline{w}z}
\end{eqnarray}
is the reproducing kernel of $H^2(m)$.

For $f\in H^2(m)$, using the reproducing property of the Szeg\"o kernel, we obtain
\begin{eqnarray*}
\langle f,K^m_w\rangle_m
&=&
\int_0^{2\pi}
f(e^{it})
\overline{
\frac{1}{h(e^{it})\overline{h(w)}}
\frac{1}{1-\overline{w}e^{it}}
}
m(t)\,dt
\\
&=&
\frac{1}{h(w)}
\int_0^{2\pi}
f(e^{it})h(e^{it})
\overline{
\frac{1}{1-\overline{w}e^{it}}
}
\,dt
\\
&=&
f(w).
\end{eqnarray*}

The BVC assertion in (IV) follows from the fact that $H^2(m)$ contains a dense subset of bounded functions and
\[
\lim_{|w|\to1}\|K^m_w\|_m=\infty.
\]
Indeed, for any $\epsilon>0$, there exists a bounded function, for instance a polynomial $g\in H^2(m)$, such that
\[
\|f-g\|_m<\epsilon.
\]
By the Cauchy--Schwarz inequality,
\[
|\langle f,E^m_w\rangle_m|
\leq
\|f-g\|_m
+
\frac{\sup_{z\in{\bD}}|g(z)|}{\|K^m_w\|_m}
<
2\epsilon
\]
when $|w|$ is sufficiently close to $1$. Owing to the uniform upper boundedness of $|h(w)|$ in ${\bD}$, we have
\[
\lim_{|w|\to1}\|K^m_w\|_m^2
=
\lim_{|w|\to1}
\frac{1}{|h(w)|^2}
\frac{1}{1-|w|^2}
=
\infty.
\]
Therefore, $H^2(m)$ satisfies BVC. The proof is complete. \qed

\def\ba{\bf a}

\begin{remark}
The BVC property in (IV) is analogous to the Riemann--Lebesgue lemma for Fourier coefficients. The importance of (IV) is that BVC implies the \emph{Maximal Selection Principle} (MSP); that is, there exists $w_0\in{\bD}$ such that
\[
|\langle f,E^m_{w_0}\rangle_m|
=
\sup
\left\{
|\langle f,E^m_w\rangle_m|:\ w\in{\bD}
\right\}.
\]
In the general setting, BVC also implies the attainability of an optimal parameter at an interior point. Such attainability not only determines an adaptive TM system, although it may not be unique, but also guarantees the existence of an approximation by a double-zero TM system. In the classical Hardy space case, the latter has been observed to be more effective than AFD; see \cite{QWQW}.
\end{remark}

\bigskip

\begin{remark}
Compared with the other properties, (IV) is not indispensable. If BVC fails, MSP may be replaced by a weaker form of MSP; see \cite{DT}. The following version is the one most commonly used. For any $\rho\in(0,1)$, there exists $w_0\in{\bD}$ such that
\begin{eqnarray}\label{sup}
|\langle f,E_{w_0}\rangle_{H_K}|
>
\rho
\sup
\left\{
|\langle f,E_w\rangle_{H_K}|:\ w\in{\bD}
\right\}.
\end{eqnarray}
This is usually called the \emph{weak Maximal Selection Principle} (W-MSP), and it also gives rise to an effective sparse representation of $f$. Although the attainability of an optimal parameter is not required in this case, a function value close to the supremum in \eqref{sup} still has to be found.
\end{remark}

\bigskip

Correspondingly, there are weaker forms of the conditions (I)--(IV), which, in general notation, are listed as follows. Let ${\mathcal H}$ be a general Hilbert space.

\medskip

\noindent
{\rm (I)$^\ast$}
In the set-theoretic sense,
\[
{\mathcal H}\subset H^2,
\]
and
\[
\|f\|_{H^2}
\leq
C\|f\|_{\mathcal H},
\qquad
\forall f\in{\mathcal H}.
\]
Hence, ${\mathcal H}$ is an RKHS and in the set-theoretic sense
\[
{\mathcal H}=H_K.
\]

\medskip

\noindent
{\rm (II)$^\ast$}
For every
\[
(a_1,\ldots,a_{n-1})\in{\bD}^{n-1},
\]
the finite Blaschke product
\[
\Psi_n(z)
=
\prod_{k=1}^{n-1}\tau_{a_k}(z)
\]
belongs to ${\mathcal H}$, and ${\mathcal H}$ forms an algebra over the complex field.

\medskip

\noindent
{\rm (III)$^\ast$}
The TM-like system
\[
{\ }^{\ba}\{B_n\}_n,
\]
defined by \eqref{read}, is orthonormal in ${\mathcal H}$.

\medskip

\noindent
{\rm (IV)$^\ast$}
The reproducing kernel of $H_K$ satisfies BVC.

   \section{TM Systems and AFD Sparse Representations in $H^2(m)$}

For every RKHS, $H_K$, one can formally define a TM-like system in the following manner. Let
${\bf a}=(a_1,\ldots,a_n,\ldots)$
be a finite or infinite sequence of complex numbers in the unit disc ${\bD}$, allowing repetitions. We define the TM-like system generated by ${\bf a}$ in $H_K$ as
\begin{eqnarray}\label{read}
\{B_n\}_n,
\qquad
\text{where}
\qquad
B_n(z)
=
E_{a_n}(z)
\prod_{k=1}^{n-1}\tau_{a_k}(z),
\end{eqnarray}
where, for any $a\in{\bD}$, $\tau_a(z)$ is the canonical M\"obius transform with pole $a$, and $E_a$ is the normalized reproducing kernel at $a$. To indicate the depende withnce on ${\ba}$, we also write $\{B_n\}_n$ as
\[
{\ }^{\ba}\{B_n\}_n,
\]
and $B_n$ as ${\ }^{\ba}B_n$, and so forth.

\begin{defn}
If ${\mathcal H}$ satisfies the conditions {\rm (I)$^\ast$}, {\rm (II)$^\ast$}, and {\rm (III)$^\ast$}, then ${\mathcal H}$ is an RKHS, ${\mathcal H}=H_K,$ and $H_K$ is equipped with kernel-TM systems.
\end{defn}

\begin{theorem}\label{weak}
\begin{enumerate}

\item
If such a defined $H_K$ kernel-TM system is a basis of $H_K$, then ${\ba}$ satisfies the non-Blaschke condition.

\item
The conditions {\rm (I)$^\ast$}, {\rm (II)$^\ast$}, {\rm (III)$^\ast$}, and {\rm (IV)$^\ast$} allow construction of an AFD made by an $H_K$ kernel-TM system.

\item
If, in addition, the reverse inequality
\[
\|\cdot\|_{H_K}
\leq
C\|\cdot\|_{H^2}
\]
holds, then ${\ }^{\ba}\{B_n\}_n$ is a basis of $H_K$ whenever ${\ba}$ satisfies the non-Blaschke condition.
\end{enumerate}
\end{theorem}

As a consequence of Theorem \ref{weak}, we obtain the following result.

\begin{theorem}
 The space  $H_{K^m}$ admits a kernel-TM system
\[
\{B_n\}_n,
\qquad
B_n(z)
=
E^m_{a_n}(z)
\prod_{k=1}^{n-1}\tau_{a_k}(z).
\]
\end{theorem}

\noindent
{\bf Proof.}
It amount to verifying the orthonormality condition (III)* of the TM-like system in $H^2(m).$ Let
\[
\{B_n\}_{n=1}^{\infty}
\]
be generated by ${\bf a}$. Since
\[
|\tau_{a_k}(e^{it})|=1,
\]
we have
\[
\begin{aligned}
\langle B_n,B_n\rangle_m
&=
\int_0^{2\pi}
E^m_{a_n}(e^{it})
\prod_{k=1}^{n-1}\tau_{a_k}(e^{it})
\overline{
E^m_{a_n}(e^{it})
\prod_{k=1}^{n-1}\tau_{a_k}(e^{it})
}
m(t)\,dt
\\
&=
\|E^m_{a_n}\|_m^2
=
1.
\end{aligned}
\]

If $n=m+l$ with $l>0$, then, using the reproducing kernel property, we obtain
\[
\begin{aligned}
\langle B_n,B_m\rangle_m
&=
\int_0^{2\pi}
E^m_{a_n}(e^{it})
\prod_{k=m}^{n-1}\tau_{a_k}(e^{it})
\overline{E^m_{a_m}(e^{it})}
m(t)\,dt
\\
&=
\left\langle
E^m_{a_n}
\prod_{k=m}^{n-1}\tau_{a_k},
E^m_{a_m}
\right\rangle_m
\\
&=
0,
\end{aligned}
\]
since
\[
\tau_{a_m}(a_m)=0.
\]

It follows directly from the norm equivalence
\[
\|f\|_{H^2}
\approx
\|f\|_m
\]
that ${\ }^{\ba}\{B_n\}_n$ is complete in $H_{K^m}$ if and only if ${\ba}$ satisfies
\[
\sum_{k=1}^{\infty}(1-|a_k|)=\infty,
\]
which is interpreted as the non-Blaschke condition. The proof is complete.\qed

Polarization of \eqref{pol} gives, for $f,g\in H_{K^m}$,
\[
\langle f,g\rangle_m
=
\langle fh,gh\rangle_{H^2({\bD})}.
\]
One may wonder whether this isometry yields anything new. It does, because, for
\[
f\in H^2({\bD}),
\]
or, more precisely, $f\in H^2(m)$, one may seek a decomposition different from classical AFD.

The corresponding $H^2(m)$-AFD begins by finding $a^m_1\in{\bD}$ such that, for $f_1=f$,
\begin{eqnarray}\label{hold}
a^m_1
=
\arg\max
\left\{
|\langle f_1,E^m_a\rangle_m|:\ a\in{\bD}
\right\},
\end{eqnarray}
where
\begin{eqnarray}\label{gi}
E^m_a(z)
=
\frac{
|h(a)|\sqrt{1-|a|^2}
}{
h(z)\overline{h(a)}(1-\overline{a}z)
}
\end{eqnarray}
is the normalized reproducing kernel.

A direct computation using the reproducing property gives
\begin{eqnarray*}
|\langle f_1,E^m_a\rangle_m|
=
|h(a)|\sqrt{1-|a|^2}\,|f_1(a)|.
\end{eqnarray*}

The BVC of $H^2(m)$ implies that there exists $a^m_1\in{\bD}$ satisfying \eqref{hold}. By the reproducing kernel property, the orthogonal complement term
\[
f_1(z)
-
\langle f_1,E^m_{a^m_1}\rangle_m
E^m_{a^m_1}(z)
\]
vanishes at $z=a^m_1$. Therefore,
\[
f_2(z)
=
\frac{
f_1(z)
-
\langle f_1,E^m_{a^m_1}\rangle_m
E^m_{a^m_1}(z)
}{
\tau_{a^m_1}(z)
}
\in H^2(m),
\]
where $\tau_{a^m_1}(z)$ is the canonical M\"obius transform with pole $a^m_1$.

We therefore obtain the first step of the $H^2(m)$-AFD decomposition:
\[
f(z)
=
\langle f_1,E^m_{a^m_1}\rangle_m
E^m_{a^m_1}(z)
+
f_2(z)\tau_{a^m_1}(z).
\]

Similarly, we decompose $f_2\in H_{K^m}$ in the same manner, and continue iteratively. This yields
\begin{eqnarray}\label{stead}
f_{k+1}(z)
=
\frac{
f_k(z)
-
\langle f_k,E^m_{a^m_k}\rangle_m
E^m_{a^m_k}(z)
}{
\tau_{a^m_k}(z)
},
\end{eqnarray}
where $a^m_k$ is an optimal parameter selected at the $k$-th step:
\begin{eqnarray}\label{under}
a^m_k
=
\arg\max
\left\{
|\langle f_k,E^m_a\rangle_m|:\ a\in{\bD}
\right\}.
\end{eqnarray}

\begin{theorem}\label{6}
Let $f\in H^2({\bD})$. Then the following expansion holds in the $H_{K^m}$-norm:
\begin{eqnarray}\label{AFD}
f(z)
=
\sum_{k=1}^{\infty}
\langle f_k,E^m_{a^m_k}\rangle_m
{\ }^{\bf a}B_k(z),
\end{eqnarray}
where the TM system
\[
\{{\ }^{\bf a}B_n\}
\]
is generated by the optimally selected parameter sequence
\[
{\bf a}
=
{\bf a}^m
=
\{a^m_n\},
\]
and
\[
{\ }^{\bf a}B_n(z)
=
E^m_{a^m_n}(z)
\prod_{l=1}^{n-1}\tau_{a^m_l}(z).
\]
\end{theorem}

The proof follows the same argument as that used for the main result in \cite{QWa1}.

The expansion given in Theorem \ref{6} possesses all the merits of AFD, including the fact that it gives rise to a decomposition with phase derivatives of describable spacial characters inherited from M\"obius transforms and the outer function $h.$ It also introduces a new functional variable, namely the weight function $m(t)$. The latter provides an additional degree of adaptivity.

We note that, owing to orthogonality, the following useful relation holds:
\[
\langle f_k,E^m_{a^m_k}\rangle_m
=
\langle f,{\ }^{\bf a}B_k\rangle_m.
\]

There is also an associated interpolation result. For any finite sequence
\[
{\bf b}=(b_1,\ldots,b_n)
\]
in $\bD$, allowing repetitions, the following algebraic identity holds:
\begin{eqnarray}\label{inter}
f(z)
&=&
\sum_{k=1}^{n}
\langle f,{\ }^{\bf b}B_k\rangle_m
{\ }^{\bf b}B_k(z)
+
{\ }^{\bf b}f_{n+1}(z)
\prod_{l=1}^{n}\tau_{b_l}(z)
\nonumber\\
&=&
P_{\{b_1,\ldots,b_n\}}^{H^2(m)}(f)(z)
+
{\ }^{\bf b}f_{n+1}(z)
\prod_{l=1}^{n}\tau_{b_l}(z),
\end{eqnarray}
where
\[
P_{\{b_1,\ldots,b_n\}}^{H^2(m)}
\]
denotes the orthogonal projection in $H^2(m)=H_{K^m}$ onto the span generated by the corresponding kernel functions. Clearly,
\[
P_{\{b_1,\ldots,b_n\}}^{H^2(m)}(f)
\]
interpolates $f$ at the points $b_1,\ldots,b_n$.

\section{Double-Zero AFD of $H^2(m)$}

The aim of the present section is to extend the double-zero AFD established in \cite{QWQW} to
$
H_{K^m}=H^2(m).
$

We will systematically use the following notation. Let
\[
f\in H^2(m),\qquad f=f_1,
\]
and define
\[
f_2(z)
=
\frac{
f_1(z)-\langle f_1,E^m_{a_1}\rangle_m E^m_{a_1}(z)
}{
\frac{z-a_1}{1-\overline{a}_1z}
},
\]
and
\[
\tilde{f}_2(z)
=
\frac{
\tilde{f}_1(z)-\langle \tilde{f}_1,E^m_{a_1}\rangle_m E^m_{a_1}(z)
}{
\left(
\frac{z-a_1}{1-\overline{a}_1z}
\right)^2
},
\]
and so forth.

\begin{lemma}\label{core22}
Let $f\in H_{K^m}$ be non-trivial and let $a_1\in\bf{D}$ satisfy
\begin{align}\label{optimal1}
|\langle f,E^m_{a_1}\rangle_m|
=
\max
\{
|\langle f,E^m_a\rangle_m|:a\in\bf{D}
\}.
\end{align}
Then
\[
g_2(z)
\triangleq
f_1(z)-\langle f_1,E^m_{a_1}\rangle_m E^m_{a_1}(z)
=
(z-a_1)^2\tilde{f}_2(z),
\]
where
\[
\tilde{f}_2\in H_{K^m}.
\]
\end{lemma}

\textit{Proof.}
We first show that
\[
f_2(z)
=
\frac{g_2(z)}
{\frac{z-a_1}{1-\overline{a}_1z}}
\]
vanishes at $z=a_1$. It suffices to prove that
\[
\lim_{z\to a_1}f_2(z)=0.
\]

For $z\neq a_1$, we have
\begin{eqnarray}\label{zero1}
f_2(z)
&=&
\frac{
f(z)-\langle f,E^m_{a_1}\rangle_m E^m_{a_1}(z)
}{
\frac{z-a_1}{1-\overline{a}_1z}
}
\nonumber\\
&=&
(1-\overline{a}_1z)
\left(
\frac{f(z)-f(a_1)}{z-a_1}
+
\frac{
f(a_1)-f(a_1)\frac{K^m_{a_1}(z)}
{\|K^m_{a_1}\|_m^2}
}{z-a_1}
\right)
\nonumber\\
&=&
(1-\overline{a}_1z)
\left(
\frac{f(z)-f(a_1)}{z-a_1}
-
\frac{f(a_1)}{\|K^m_{a_1}\|_m^2}
\frac{K^m_{a_1}(z)-K^m_{a_1}(a_1)}
{z-a_1}
\right)
\nonumber\\
&\to&
(1-|a_1|^2)
\left(
f'(a_1)
-
\frac{f(a_1)}{\|K^m_{a_1}\|_m^2}
\frac{\partial K^m_{a_1}}{\partial z}(a_1)
\right),
\qquad z\to a_1 .
\end{eqnarray}

We now show that the last expression is zero. Due to the optimality of $a_1$ in the open disc,
\begin{eqnarray}\label{shown1}
0
&=&
\frac{\partial}{\partial z}
|\langle f,E^m_z\rangle_m|^2
\nonumber\\
&=&
\frac{\partial}{\partial z}
\left(
\frac{f(z)\overline{f(z)}}{\|K^m_z\|_m^2}
\right)
\nonumber\\
&=&
\frac{\overline{f(z)}}{\|K^m_z\|_m^2}
\left(
f'(z)
-
\frac{f(z)}{\|K^m_z\|_m^2}
\frac{\partial}{\partial z}K_z^m(z)
\right).
\end{eqnarray}

Comparing \eqref{zero1} and \eqref{shown1}, it remains to verify that
\begin{eqnarray}\label{veri}
\left.
\frac{\partial K^m_{a_1}(z)}{\partial z}
\right|_{z=a_1}
=
\left.
\frac{\partial}{\partial z}
K_z^m(z)
\right|_{z=a_1}.
\end{eqnarray}

Since
\[
K^m_w(z)
=
\frac{1}{h(z)\overline{h(w)}}
\frac{1}{1-\overline{w}z},
\]
we have
\begin{eqnarray*}
\frac{\partial}{\partial z}K^m(z,a_1)
&=&
\frac{1}{\overline{h(a_1)}}
\frac{\partial}{\partial z}
\left(
\frac{1}{h(z)}
\frac{1}{1-\overline{a_1}z}
\right)
\\
&=&
\frac{1}{\overline{h(a_1)}}
\frac{-h'(z)}{h^2(z)}
\frac{1}{1-\overline{a_1}z}
+
\frac{1}{\overline{h(a_1)}}
\frac{1}{h(z)}
\frac{\overline{a_1}}
{(1-\overline{a_1}z)^2}.
\end{eqnarray*}

On the other hand,
\begin{eqnarray*}
\frac{\partial}{\partial z}K^m(z,z)
&=&
\frac{\partial}{\partial z}
\left(
\frac{1}{\overline{h(z)}}
\frac{1}{h(z)}
\frac{1}{1-\overline{z}z}
\right)
\\
&=&
\frac{1}{\overline{h(z)}}
\left(
\frac{-h'(z)}{h^2(z)}
\frac{1}{1-\overline{z}z}
+
\frac{1}{h(z)}
\frac{\overline{z}}
{(1-\overline{z}z)^2}
\right).
\end{eqnarray*}

The above two expressions coincide at $z=a_1$. Hence \eqref{veri} holds, and the proof of the lemma is complete.\qed

Define, for $a_1$ satisfying \eqref{optimal1},
\[
\tilde{f}_2(z)
=
\frac{
g_2(z)
}{
\left(
\frac{z-a_1}{1-\overline{a}_1z}
\right)^2
}
\in H_{K^m},
\]
where the well-definedness is justified by Lemma \ref{core22}.

Owing to Lemma \ref{core22}, we can proceed with the corresponding iteration and obtain the following theorem.

\begin{theorem}
Let $f\in H_{K^m}=H^2(m)$ be given. With the optimal selections
\[
a_k
=
\arg\max_{a\in{\bD}}
\{
|\langle \tilde{f}_k,E^m_a\rangle_m|^2
\},
\]
the sequence
\[
\{\tilde{B}^{\ba}_k\}_{k=1}^{\infty}
\]
is an orthonormal system, called the \emph{double-zero kernel-TM system}, and
\[
f(z)
=
\sum_{k=1}^{\infty}
\langle \tilde{f}_k,E^m_{a_k}\rangle_m
\tilde{B}^{\ba}_k(z)
\]
holds in the space $H_{K^m}$, where
\[
E^m_a(z)
=
\frac{K^m_a(z)}
{\|K^m_a\|_m},
\]
\[
\tilde{B}^{\ba}_n(z)
=
\frac{K^m_{a_n}(z)}
{\|K^m_{a_n}\|_m}
\prod_{l=1}^{n-1}
\left(
\frac{z-a_l}
{1-\overline{a}_lz}
\right)^2,
\]
and
\[
\tilde{f}_n(z)
=
\frac{
g_n(z)
}{
\prod_{l=1}^{n-1}
\left(
\frac{z-a_l}
{1-\overline{a}_lz}
\right)^2
}.
\]
Here $g_n$ is the projection of $f$ onto the orthogonal complement of
\[
{\rm span}\{\tilde{B}^{\ba_l}_l\}_{l=1}^{n-1},
\]
namely,
\[
g_n
=
f-
\sum_{l=1}^{n-1}
\langle f,\tilde{B}^{\ba_l}_l\rangle_m
\tilde{B}^{\ba_l}_l .
\]
\end{theorem}

As in the classical Hardy space case, the double interpolation result is valid, and only valid, for the attained optimal parameters
\[
\tilde{\ba}
=
\{\tilde a_1,\ldots,\tilde a_n\},
\qquad
\tilde a_1=a_1 .
\]
For other related results, including interpolation on the boundary and $n$-best double-zero AFD, we refer the reader to \cite{QWQW}.

Stochastic AFD (SAFD) is available for random signals in Bochner-Boundary-Weighted Hardy spaces. However, stochastic DAFD does not exist in Bochner-Boundary-Weighted Hardy spaces, since optimal parameters cannot in general be attained at the same points for a group of random signals. 

\section{Justification of RKHSs as Subclasses of $H^2({\bD})$ Against Kernel-TM System}

The purpose of this section is to show, through the analysis of invalidity of conditions (I)$^\ast$ to (IV)$^\ast$, that kernel-TM systems and associated AFD (\ref{read}) are not directly available for some commonly studied RKHSs as merely subclasses of $H^2({\bD})$.

\noindent {\bf Hardy-Sobolev Spaces}

There is another type of \emph{weighted Hardy spaces} (\cite{qu2018,qu2019}) in which the weights are assigned to the Fourier coefficients. They are commonly called \emph{Hardy-Sobolev spaces}. Let
\[
H^2_\beta({\bD})
=
\left\{
f:{\bD}\to{\bf C}:
f(z)=\sum_{k=0}^{\infty}c_kz^k,\ 
\|f\|^2_{H^2_\beta}
=
\sum_{k=0}^{\infty}(k+1)^\beta |c_k|^2<\infty
\right\}.
\]
When $\beta<0$, the same definition gives an equivalent form of weighted Bergman spaces. Bergman space functions are not contained in the Hardy space and generally do not possess non-tangential boundary limits. Therefore, we restrict our discussion to
\[
0\leq\beta<\infty.
\]
The case $\beta=0$ corresponds to the Hardy space, for which conditions (I) to (IV) are valid.

\begin{theorem}
For the whole range
\[
0\leq\beta<\infty,
\]
the Hardy-Sobolev spaces $H^2_\beta({\bD})$ satisfy (I)$^\ast$ and (II)$^\ast$, but not (III)$^\ast$. Moreover, $H^2_\beta({\bD})$ satisfies (IV)$^\ast$ if and only if
\[
0\leq\beta\leq1.
\]
\end{theorem}

\noindent{\bf Proof}
By polarization, the associated inner product is
\[
\langle f,g\rangle_{H^2_\beta}
=
\sum_{k=0}^{\infty}
(k+1)^\beta c_k\overline{b_k},
\]
where $\{c_k\}$ and $\{b_k\}$ are, respectively, the Fourier coefficients of $f$ and $g$.

For $\beta>0$, we have
\[
\|f\|_{H^2}
\leq
\|f\|_{H^2_\beta},
\]
and hence, in the sense of set inclusion,
\[
H^2_\beta\subset H^2.
\]
However, $H^2_\beta$ is not a closed subspace of $H^2$.

Now we verify (II)$^\ast$. Since a finite Blaschke product is holomorphic in a neighborhood of the closed unit disc, its Fourier coefficients decay exponentially. Therefore,
\[
Bf\in H^2_\beta
\]
for every $f\in H^2_\beta$ and every finite Blaschke product $B$. We omit the tedious details.

The reproducing kernel can be expanded as
\begin{eqnarray}
k^\beta(z,w)
=
\sum_{k=0}^{\infty}
(k+1)^{-\beta}(z\overline w)^k
=
\sum_{k=1}^{\infty}
\phi_k(z)\overline{\phi_k(w)},
\end{eqnarray}
where
\[
\{\phi_k(z)\}_{k=1}^{\infty}
=
\left\{
\left(\frac{1}{1+k}\right)^{\frac{\beta}{2}}z^k
\right\}_{k=1}^{\infty}
\]
is an orthonormal basis.

The squared norm of the kernel satisfies
\[
\|k_w\|^2_{H^2_\beta}
=
k^\beta(w,w)
=
\sum_{k=1}^{\infty}
\frac{1}{(1+k)^\beta}|w|^{2k}
\to\infty
\]
as $|w|\to1$ for
\[
0\leq\beta\leq1,
\]
and hence BVC holds for these values of $\beta$. In \cite{qu2018}, it is shown that BVC does not hold for $\beta>1$.

Now we derive the norm and inner product representations in integral forms on the boundary. Let
\[
D=\frac{1}{i}\frac{d}{dt}.
\]
Then the alternative norm and inner product expressions in terms of the non-tangential boundary limit functions on the unit circle are given by
\[
\|f\|_{H^2_\beta}^2
=
\frac1{2\pi}
\int_0^{2\pi}
|(I+D)^{\beta/2}f^\ast|^2dt,
\]
and
\[
\langle f,g\rangle_{H^2_\beta}
=
\frac1{2\pi}
\int_0^{2\pi}
(I+D)^{\beta/2}f^\ast
\overline{(I+D)^{\beta/2}g^\ast}
dt.
\]
This shows that when $\beta>0$, the space $H^2_\beta$ does not satisfy (III)$^\ast$. The proof is complete.\qed

The Hardy-Sobolev reproducing kernels are identical to the inverse fractional differential operator
\[
(I+D)^\beta
\]
applied to the Szeg\"o kernel. That is,
\[
K^\beta_w(e^{it})
=
(I+D)^{-\beta}
\left(
\frac{1}{1-\overline w e^{it}}
\right).
\]

We note that the Hardy-Sobolev space with $\beta=1$ is the so-called Dirichlet space corresponding to the Sobolev derivative of order $1/2$. For
\[
\beta\in[0,1],
\]
BVC holds and therefore POAFD is valid. For $\beta>1$, since BVC fails, POAFD cannot be directly applied. However, as already noted, weak POAFD is always available.

\medskip

\noindent{\bf Model Spaces}

Let $\Theta$ be an inner function, that is,
\[
\Theta\in H^\infty,\qquad
|\Theta(e^{it})|=1, \quad {\rm a.\ e.},
\]
for almost every $e^{it}\in\partial{\bD}$.

The corresponding \emph{model space} is defined as
\[
K_\Theta=H^2\ominus \Theta H^2.
\]

\begin{theorem}
The model spaces satisfy (I)$^\ast$, (III)$^\ast$, and (IV)$^\ast$. Moreover, $K_\Theta$ satisfies (IV)$^\ast$ if and only if
\[
\theta'(t)=\infty
\]
for every $e^{it}\in\partial{\bD}$, where
\[
\Theta(e^{it})=e^{i\theta(t)}.
\]
The space $K_\Theta$ does not satisfy (II)$^\ast$, and hence no kernel-TM system of the type (\ref{read}) is available.
\end{theorem}

\noindent{\bf Proof}
For a fixed $\Theta$, the model space $K_\Theta$ consists of those Hardy functions $f\in H^2$ that are orthogonal to all functions of the form
\[
\Theta h,\qquad h\in H^2.
\]
Equivalently,
\[
f\in K_\Theta
\Longleftrightarrow
f\in H^2
\quad\text{and}\quad
P_+(\overline\Theta f)=0.
\]

$K_\Theta$ is a closed subspace of $H^2$, and therefore (I)$^\ast$ and (III)$^\ast$ hold. It inherits the inner product of $H^2$, under which its reproducing kernel is
\[
k^\Theta_w(z)
=
\frac{
1-\overline{\Theta(w)}\Theta(z)
}{
1-\overline wz
}.
\]

Taking
\[
f=k^\Theta_w
\]
for a fixed $w\in{\bD}$ and $B$ to be any non-trivial finite Blaschke product, we find that $fB$ is generally no longer in $K_\Theta$. Hence (II)$^\ast$ does not hold.

The relation
\begin{eqnarray}\label{from}
\|k^\Theta_w\|^2_{K_\Theta}
=
\frac{1-|\Theta(w)|^2}{1-|w|^2}
\end{eqnarray}
implies that BVC holds if and only if
\[
\lim_{|w|\to1}
\frac{1-|\Theta(w)|^2}{1-|w|^2}
=
\infty
\]
at all boundary points.

This condition is equivalent to the fact that $\Theta$ has no Ahern-Clark points (finite boundary phase derivatives). It is also equivalent to
\[
\theta'(t)=\infty,\qquad t\in[0,2\pi),
\]
where
\[
\theta'(t)
=
\sum_n
\frac{1-|a_n|^2}{|e^{it}-a_n|^2}
+
2
\int_{\partial{\bD}}
\frac{d\nu(\xi)}
{|e^{it}-\xi|^2},
\]
and
\[
\Theta(e^{it})=e^{i\theta(t)}.
\]
Here $d\nu$ is the singular measure defining the singular inner function part, and $\{a_n\}$ are the zeros of $\Theta$; see \cite{AC,FM,QianPhase}. The proof is complete.\qed

\medskip

\noindent{\bf The De Branges--Rovnyak Spaces}

As RKHSs, the De Branges--Rovnyak spaces $\mathcal H(b)$ for
\[
b\in H^\infty,\qquad
\|b\|_{H^\infty}\leq1
\]
are defined by the positive definite kernel
\[
k^b_w(z)
=
\frac{
1-\overline{b(w)}b(z)
}{
1-\overline wz
}.
\]
They are generalizations of model spaces. We only consider a simple and particular case.

\begin{theorem}
If
\[
|b(w)|<1
\]
at all points $w\in{\bD}$, then $\mathcal H(b)$ satisfies (I)$^\ast$ and (IV)$^\ast$, but not (II)$^\ast$ and (III)$^\ast$.
\end{theorem}

\noindent{\bf Proof}
Obviously, $\mathcal H(b)$ is not invariant under multiplication by finite Blaschke products. As sets of functions,
\[
\mathcal H(b)=H^2,
\]
but $\mathcal H(b)$ is not a closed subspace of $H^2$; rather, it is a Hilbert space contractively contained in $H^2$. Therefore, (I)$^\ast$ holds.

In particular,
\[
\langle f,g\rangle_{\mathcal H(b)}
=
\int_{\partial{\bD}}
f^\ast\overline{g^\ast}dt
+
\int_{\partial{\bD}}
P_+(\overline{\phi^\ast}f^\ast)
\overline{
P_+(\overline{\phi^\ast}g^\ast)
}
dt,
\]
where
\[
\phi=\frac ba,
\]
$a$ is an outer function in $H^\infty$ satisfying
\[
|a|^2+|b|^2=1,
\]
and $P_+$ is the Hardy space projection; see \cite{FM}. This integral representation of the inner product is not a simple weighted $L^2$ integral and therefore does not satisfy the requirement (III)$^\ast$.

It satisfies (IV)$^\ast$, because
\[
\lim_{|w|\to1}
\|k^b_w\|^2
=
\lim_{|w|\to1}
\frac{1-|b(w)|^2}{1-|w|^2}
=
\infty.
\]
The proof is complete.\qed

\section*{Acknowledgements}

This work was supported by the Prevention and Control of Emerging and Major Infectious Diseases-National Science and Technology Major Project (Grant No. 2027ZD01999504), the Major Project of Guangzhou National Laboratory (Grant No. GZNL2024A01004), and the Science and Technology Development Fund of Macau SAR (Grant No. 0015/2025/AFJ)

\end{document}